\documentclass[12pt]{article}
\usepackage{graphicx}
\usepackage{amsmath, amsthm}
\usepackage{amsfonts}
\usepackage{amssymb}
\usepackage{url}
\usepackage{dsfont} 
\usepackage[mathlines]{lineno}

\usepackage{verbatim}
\usepackage{color}
\definecolor{red}{rgb}{1,0,0}

\usepackage{hyperref}
\usepackage{dsfont}
\usepackage{tikz}
\usepackage{rotating}
\usepackage{tkz-graph}
\usepackage{graphicx}
\usepackage{caption}
\usepackage{subcaption}

\tikzstyle{vertex}=[circle, draw, inner sep=0pt, minimum size=6pt]

\tikzstyle{rvertex}=[circle, red, fill, draw, inner sep=0pt, minimum size=6pt]

\tikzstyle{gvertex}=[circle, green, fill, draw, inner sep=0pt, minimum size=6pt]

\tikzstyle{bvertex}=[circle, blue, fill, draw, inner sep=0pt, minimum size=6pt]

\tikzstyle{Bvertex}=[circle, black, fill, draw, inner sep=0pt, minimum size=6pt]

\newcommand{\vertex}{\node[vertex]}

\newcommand{\Bvertex}{\node[Bvertex]}

\usetikzlibrary{calc, shapes, backgrounds} 

\def\noi{\noindent}

\newcommand{\mr}{\operatorname{mr}}

\newcommand{\PC}{\operatorname{P}}

\newcommand{\TC}{\operatorname{T}}

\newcommand{\tw}{\operatorname{tw}}

\newcommand{\mrp}{\operatorname{mr}_+}

\newcommand{\Mp}{\operatorname{M}_+}

\newcommand{\ZFN}{\operatorname{Z}}

\newcommand{\Zp}{\operatorname{Z}_+}

\newcommand{\nullity}{\operatorname{null}}

\newtheorem{thm}{Theorem}
\newtheorem{defn}[thm]{Definition}
\newtheorem{prop}[thm]{Proposition}
\newtheorem{cor}[thm]{Corollary}
\newtheorem{lem}[thm]{Lemma}

\newtheorem{conj}[thm]{Conjecture}

\newtheorem{obs}[thm]{Observation}

\newcommand{\bit}{\begin{itemize}}
	
	\newcommand{\eit}{\end{itemize}}

\newcommand{\ben}{\begin{enumerate}}
	
	\newcommand{\een}{\end{enumerate}}

\newcommand{\beq}{\begin{equation}}

\newcommand{\eeq}{\end{equation}}

\newcommand{\bea}{\begin{eqnarray*}}
	
	\newcommand{\eea}{\end{eqnarray*}}

\newcommand{\bpf}{\begin{proof}}
	
	\newcommand{\epf}{\end{proof}}

\title{On the tree cover number and the positive semidefinite maximum nullity of a graph}

\author{Chassidy Bozeman\thanks{Department of Mathematics and Statistics, Mount Holyoke College,  South Hadley, MA 01075, USA (cbozeman@mtholyoke.edu)}}

\date{}

\begin{document}
	
	\maketitle {}

	
	
	%
	
	
	
	
	
	
	
	%
	
	
	
	
	
	
	

	\abstract {For a simple graph $G=(V,E),$ let $\mathcal{S}_+(G)$ denote the set of real positive semidefinite matrices $A=(a_{ij})$ such that $a_{ij}\neq 0$ if $\{i,j\}\in E$ and $a_{ij}=0$ if $\{i,j\}\notin E$. The maximum positive semidefinite nullity of $G$, denoted $\Mp(G),$ is $\max\{\nullity(A)|A\in \mathcal{S}_+(G)\}.$ A tree cover of $G$ is a collection of  vertex-disjoint simple trees occurring as induced subgraphs of $G$ that cover all the vertices of $G$. The tree cover number of $G$, denoted $T(G)$, is the cardinality of a minimum tree cover. It is known that the tree cover number of a graph and the maximum positive semidefinite nullity of a graph are equal for outerplanar graphs, and it was conjectured in 2011 that $T(G)\leq \Mp(G)$ for all graphs [Barioli et al., Minimum semidefinite rank of outerplanar graphs and the tree cover number, {\em Elec. J. Lin. Alg.,} 2011]. We show that the conjecture is true for certain graph families. Furthermore, we prove bounds on $T(G)$ to show that if $G$ is a connected outerplanar graph on $n\geq 2$ vertices, then $\Mp(G)=T(G)\leq \left\lceil\frac{n}{2}\right\rceil$, and if $G$ is a connected outerplanar graph on $n\geq 6$ vertices with no three or four cycle, then $\Mp(G)=T(G)\leq \frac{n}{3}$. We also characterize connected outerplanar graphs with $\Mp(G)=T(G)=\left\lceil\frac{n}{2}\right\rceil.$}
	
\section{Introduction}  A {\em graph} is a pair $G=(V,E)$ where $V$ is the vertex set and $E$ is the set of edges (two element subsets of the vertices). All graphs discussed are simple (no loops or multiple edges) and finite. For a graph $G=(V,E)$ on $n$ vertices, we use $\mathcal{S}_+(G)$ to denote the set of real $n\times n$  positive semidefinite matrices $A=(a_{ij})$ satisfying $a_{ij}\neq 0$ if and only if $\{i,j\}\in E$, for $i\neq j$, and $a_{ii}$ is any non negative real number. The maximum positive semidefinite nullity of $G$, denoted $\Mp(G),$ is defined as $\max\{\nullity(A) |A\in \mathcal{S}_+(G)\}$. The minimum positive semidefinite rank of $G$, denoted $\mr_+(G),$ is defined as $\min\{{\rm rank}(A) |A\in \mathcal{S}_+(G)\},$ and it follows from the Rank-Nullity Theorem that $\Mp(G)+\mrp(G)=n$. Barioli et al. \cite{BFMN11} define a {\em tree cover} of $G$ to be a collection of vertex-disjoint simple trees occurring as induced subgraphs of $G$ that cover all the vertices of $G$. The {\em tree cover number} of $G$, denoted $T(G)$, is the cardinality of a minimum tree cover, and it is used as a tool for studying the positive semidefinite maximum nullity of $G$. (In their paper \cite{BFMN11}, $G$ is allowed to be a multigraph, but we restrict ourselves to simple graphs.) It was conjectured in \cite{BFMN11} that $T(G)\leq \Mp(G)$ for all graphs, and it is shown there that $T(G)=\Mp(G)$ for outerplanar graphs.  

We show that  $T(G)\leq \Mp(G)$ for certain families of graphs in Section \ref{sectionconjecturetrue}. In Section \ref{sectionbounds}, we study $T(G)$ for connected graphs with girth at least 5 and deduce bounds on $\Mp(G)$ for connected outerplanar graphs with girth at least 5.  In section \ref{sectionouterplanar}, we characterize connected outerplanar graphs on $n$ vertices having positive semidefinite maximum nullity and tree cover number equal to the upper bound of $\left\lceil\frac{n}{2}\right\rceil$. 

\subsection{Graph theory terminology} 

For a graph $G=(V,E)$ and $v\in V$, the {\em neighborhood} of $v$, denoted $N(v),$ is the set of vertices adjacent to $v$. The degree of $v$ is the cardinality of $N(v)$ and is denoted by $\deg(v)$. A vertex of degree one is called a {\em leaf}. A set $S\subseteq V$ is {\em independent} if no two of the vertices of $S$ are adjacent. 

The {\em path} $P_n$ is the graph with vertex set $\{v_1,\ldots,v_n\}$ and edge set $\{\{v_i,v_{i+1}\}| i\in \{1,\ldots,n-1\}\}$. The {\em cycle} $C_n$ is formed by adding the edge $\{v_n,v_1\}$ to $P_n$. The {\em girth} of a graph is the size of the smallest cycle in the graph. We denote the graph on $n$ vertices containing every edge possible by $K_n$, and we use $K_{s,t}$ to denote the {\em complete bipartite graph}, the graph whose vertex set may be partitioned into two independent sets $X$ and $Y$ such that $|X|=s, |Y|=t$, for each $x\in X$ and $y\in Y$, $\{x,y\}$ is an edge, and each edge has one endpoint in $X$ and one endpoint in $Y$.  The graph $K_{1,3}$ is referred to as a {\em claw} and more generally, the graph $K_{1,t}$ is called a {\em star}.

For a graph $G=(V,E)$, a  graph $G'=(V',E')$ is a {\em subgraph } of $G$ if $V'\subseteq V$ and $E'\subseteq E$. A subgraph $G'$ is an {\em induced subgraph} of $G$ if $V(G')\subseteq V(G)$ and $E(G')=\{\{u,v\}| \{u,v\}\in E(G) \text{ and } u,v\in V(G')\}$. If $S\subseteq V(G)$, then we use $G[S]$ to denote the subgraph induced by $S$. For $S\subseteq V(G),$ we use $G-S$ to denote $G[V(G)\setminus S]$, and for $e\in E$, $G-e$ denotes the graph obtained by deleting $e$. For a graph $G$ and an induced subgraph $H$, $G-H$ denotes the graph that results from $G$ by deleting $V(H)$. A graph $H=(V(H),E(H))$ is a {\em clique} if for each $u,v\in V(H), \{u,v\}\in E(H).$ The {\em clique number} of $G$, denoted $\omega(G)$, is $\omega(G)=\max\{|V(H)|: H \text{ is a subgraph of } G \text{ and } H \text{ is a clique}\}.$ The {\em independence number} of $G$, denoted $\alpha(G)$, is the cardinality of a maximum independent set. 

A graph is {\em connected} if there is a path from any vertex to any other vertex. For a connected graph $G=(V,E)$, an edge $e\in E$ is called a {\em bridge} if $G-e$ is disconnected. We {\em subdivide} an edge $e=\{u,w\}\in E$ by removing $e$ and adding a new vertex $v_e$ such that $N(v_e)=\{u,w\}.$  

A graph $G=(V,E)$ is {\em outerplanar} if it has a crossing-free embedding in the plane with every vertex on the boundary of the unbounded face. A {\em cut-vertex} of a connected graph $G=(V,E)$ is a vertex $v\in V$ such that $G-v$ is disconnected. A graph is {\em nonseparable} if it is connected and does not have a cut-vertex. A {\em block} is a maximal nonseparable induced subgraph, and $G$ is a block-clique graph if every block is a clique.  

Throughout this paper, given a graph $G=(V,E)$ and a tree cover $\mathcal{T}$ of $G$, we use $T_v\in \mathcal{T}$ to denote the tree containing $v\in V$. 

\subsection{Preliminaries} In this section, we give some preliminary results that will be used throughout the remainder of the paper. 

It is shown in Propostion 3.3 of \cite{BFMN11} that deleting a leaf of a graph does not affect the tree cover number of the graph and that subdividing an edge does not affect the tree cover number of the graph. These two facts will be used repeatedly in the proofs throughout this paper. 

\begin{thm}\cite{vander}
	Suppose $G_i, i=1,\ldots,h$ are graphs, there is a vertex $v$ for all $i\neq j$, $G_i\cap G_j=\{v\}$ and $G=\cup_{i=1}^{h} G_i.$ Then,
	\[\Mp(G)=\left(\sum\limits_{i=1}^{h}\Mp(G_i)\right)-h+1.\]
\end{thm}

This is known as the {\em cut-vertex reduction formula.} The authors of \cite{EEHHH} give an analogous cut-vertex reduction formula for computing the tree cover number and we use this technique  multiple times throughout this paper.
\begin{prop}\label{cutvertextree}\cite{EEHHH}
	Suppose $G_i, i=1,\ldots,h$, are graphs, there is a vertex $v$ for all $i\neq j$, $G_i\cap G_j=\{v\}$ and $G=\cup_{i=1}^{h} G_i.$ Then \[T(G)=\left(\sum\limits_{i=1}^{h}T(G_i)\right)-h+1.\]
\end{prop}

In the case that $h$ from Proposition \ref{cutvertextree} is two, we say $G$ is the {\em vertex-sum} of $G_1$ and $G_2$ and write $G=G_1\oplus_v G_2$. 

Bozeman et al. \cite{REU2015} give a bound on the tree cover number of a graph in terms of  the independence number of the graph.
\begin{prop}\label{treeindep}\cite{REU2015}
	Let $G = (V, E)$ be a connected graph, and let $S\subseteq V(G)$ be an independent set. Then, $T(G) \leq|G|-|S|$. In particular, $T(G)\leq |G|-\alpha(G)$, where $\alpha(G)$ is the independence number of G. Furthermore, this bound is tight.
\end{prop}

It is also shown in Proposition 6 of \cite{REU2015} that for a graph $G=(V,E)$ and a bridge $e\in E$, $e$ belongs to some tree in every minimum tree cover. Embedded in the proof of this proposition is the following lemma, and we include the proof of the lemma for completeness.

\begin{lem}\label{treebridge}
	Let $G=(V,E)$ be a connected graph and $e=\{u,v\}$ a bridge in $E$. Let $G_1$ and $G_2$ be the connected components of $G-e$. Then $T(G)=T(G_1)+T(G_2)-1=T(G-e)-1$.
\end{lem}

For a graph $G=(V,E)$ and an edge $e\in E$, $\Mp(G)-1\leq\Mp(G-e)\leq \Mp(G)+1$ \cite{EEHHH} and $T(G)-1\leq T(G-e)\leq T(G)+1$ \cite{REU2015}. It is also known that for $v\in V$, $\Mp(G)-1\leq\Mp(G-v)\leq \Mp(G)+\deg(v)-1$ (see Fact 11 of page 46-11 of \cite{handbook}). We show that an analogous bound holds for $T(G).$
\begin{prop}
	For a graph $G=(V,E)$ and vertex $v\in V$, \[T(G)-1\leq T(G-v)\leq T(G)+\deg(v)-1.\]
\end{prop}

\begin{proof}
	Since any tree cover of $G-v$ together with the tree consisting of the single vertex $v$ is a tree cover for $G$, then  $T(G)\leq T(G-v)+1$, which gives the lower bound.  To see the upper bound, let $E_v$ denote the set of edges incident to $v$, and let $G-E_v$ denote the graph resulting from deleting the edges in $E_v$.  Note that $|E_v|=\deg(v)$, and that $T(G-E_v)=T(G-v)+1.$ Since the deletion of an edge can raise the tree cover number by at most 1, then $T(G-v)+1=T(G-E_v)\leq T(G)+\deg(v),$ and the upper bound follows. 
\end{proof}

\section{Graphs with $T(G)\leq \Mp(G)$}\label{sectionconjecturetrue} In this section, we prove that $T(G)\leq \Mp(G)$ for certain line graphs, for $G^{\triangle}$ (defined below) where $G$ is any graph, for graphs whose complements have sufficiently small tree-width, and for graphs with a sufficiently large number of edges.

We first show that for any connected graph $G$ on $n\geq 2$ vertices, $T(G)\leq \left\lceil\frac{n}{2}\right\rceil.$ 

\begin{lem}\label{deletestar}
	Let $G$ be a connected graph on $n\geq 3$ vertices. Then there exists an induced subgraph $H$ of $G$ such that $H=K_{1,p}$ for some $p\geq 1$ and $G-H$ is connected. {\rm (See Figure \ref{ffff}).} 
\end{lem}

\begin{proof}
	We prove the lemma by induction. For $n=3$ the claim holds. Let $G$ be a graph on $n\geq 4$ vertices and suppose the lemma holds for all graphs on  $3\leq k \leq n-1$ vertices. It is known that every connected graph has at most $n-2$ cut vertices (since a spanning tree of the graph has at least two leaves and the removal of these leaves will not disconnect the graph). Let $v$ be a vertex in $V(G)$ that is not a cut vertex. By hypothesis, there exists an induced subgraph $H'=K_{1,p}$  for some $p\geq 1$ in $G-v$ whose deletion does not disconnect $G-v$. First we consider the case with $p=1$, and then we consider the case with $p\geq 2$. 
	
	Case 1: Suppose $p=1$ (i.e., $H'=K_2$), and let $a,b$ be the vertices of $H'$. If $v$ has a neighbor in $G[V(G)\setminus\{a,b\}]$, then $G[V(G)\setminus\{a,b\}]$ is connected, and the claim holds with $H=H'$. Otherwise $v$ has a neighbor in $\{a,b\}$. Assume first that $v$ is adjacent to exactly one of $a$ and $b$. Without loss of generality, suppose $v$ is adjacent to $a$ and not adjacent to $b$. Then, $H=G[a,b,v]=K_{1,2}$ and $G-H$ is connected. Now suppose that $v$ is adjacent to both $a$ and $b$. Since $G-v$ is connected, then either $a$ or $b$ has a neighbor in $G[V(G)\setminus\{v,a,b\}]$. Without loss of generality, let $a$ have a neighbor in $G[V(G)\setminus\{v,a,b\}]$. Then $H=G[\{v,b\}]=K_{1,1}$ and $G-H$ is connected.  
	
	Case 2: Suppose $p\geq 2$. If $v$ has a neighbor in $G[V(G)\setminus V(H')]$, then set $H=H'$ and the claim holds. Otherwise $v$ has neighbors only in $V(H')$. Recall that $H'$ is a star. First suppose that $v$ is adjacent to a leaf $w\in V(H')$. If $w$ is not a cut vertex of $G-v$, then for $H=G[\{v,w\}]=K_{1,1}$, $G-H$ is connected. If $w$ is a cut vertex of $G-v$, then $w$ has a neighbor in $G[V(G)\setminus\{V(H')\cup v\}]$. Then  $H=G[V(H') \setminus\{w\}]=K_{1,q}$ for some $q\geq 1$ and $G-H$ is connected. Next suppose that $v$ is not adacent to a leaf in $H'$. Then it must be adjacent to the center vertex. Then $H=G[V(H')\cup \{v\}]$ is a  star, and $G-H$ is connected.  This completes the proof.
\end{proof}

\begin{figure}[h!]
	\begin{center}
		\begin{tikzpicture}[scale=1.0]
		
		\Bvertex (1) at (0,0){};
		\Bvertex (2) at (0,-1){};
		\vertex (3) at  (0,-2) {};
		\vertex (4) at (0,-3) {};
		\vertex (5) at (0,-4) {};
		\vertex (6) at (1,0){};
		\vertex (7) at (1,-1){};
		\vertex (8) at  (1,-2) {};
		\vertex (9) at (1,-3) {};
		\vertex (10) at (1,-4) {};
		
		\Edge(1)(2)
		\Edge(2)(3)
		\Edge(3)(4)
		\Edge(4)(5)
		\Edge(6)(7)
		\Edge(7)(8)
		\Edge(8)(9)
		\Edge(9)(10)
		\Edge(2)(7)
		\Edge(3)(8)
		\Edge(4)(9)
		
		\Bvertex (11) at (5,-2){};
		\vertex (12) at (4,-2){};
		\Bvertex (13) at (6,-2){};
		\Bvertex (14) at (5,-1){};
		\Bvertex (15) at (5,-3){};
		\Bvertex(16) at (4,-1){};
		\Bvertex(17) at (6,-1){};
		\Bvertex (18) at (4,-3){};
		\Bvertex (19) at (6,-3){};
		
		\Edge(11)(12)
		\Edge(11)(13)
		\Edge(11)(14)
		\Edge(11)(15)
		\Edge(11)(16)
		\Edge(11)(17)
		\Edge(11)(18)
		\Edge(11)(19)

		\end{tikzpicture}
		\caption{Two examples of Lemma \ref{deletestar}, where induced subgraphs $H$ are black.}
		\label{ffff}
	\end{center}
\end{figure}
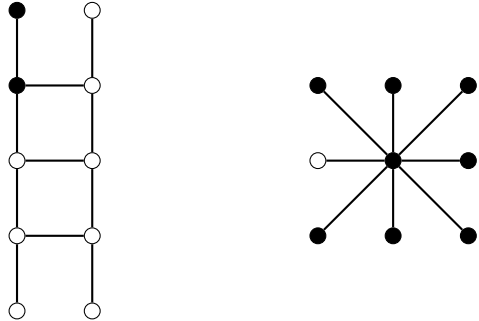
\begin{thm}\label{treecovernover2}
	For any simple connected graph $G=(V,E)$ on $n\geq 2$ vertices, $T(G)\leq \left \lceil \frac{n}{2} \right \rceil.$ 
\end{thm}

\begin{proof}
	The theorem holds for $n=2$. Let $G$ be a graph on $n\geq 3$ such that the claim holds on all graphs with fewer than $n$ vertices. By Lemma \ref{deletestar}, there exists an induced tree $H=K_{1,p}$, for some $p\geq 1$, of $G$ such that $G'=G-H$ is connected. By the induction hypothesis, $T(G')\leq \left \lceil\frac{|G'|}{2} \right\rceil\leq \left \lceil \frac{n-2}{2} \right \rceil$. Then $T(G)\leq T(G')+1\leq \left \lceil \frac{n}{2}\right \rceil.$
\end{proof}
It is shown in \cite{Louis} that for a triangle-free graph $G$, $\Mp(G)\leq \frac{n}{2}$. The next corollary is a result of Theorem \ref{treecovernover2} and the fact that $\Mp(G)=\TC(G)$ \cite{BFMN11} for outerplanar graphs.

\begin{cor}
	If $G$ is a connected outerplanar graph on $n$ vertices, then $\Mp(G)\leq \left\lceil\frac{n}{2}\right\rceil$. 
\end{cor}
Some examples of graphs with  $T(G)=\left\lceil\frac{n}{2}\right\rceil$ are the complete graphs $K_n$ and the well known Friendship graphs (graphs on $n=2k+1$ vertices, $k\geq 1$,  consisting of exactly $k$ triangles all joined at a single vertex.) Connected outerplanar graphs having $T(G)=\left\lceil\frac{n}{2}\right\rceil$ are characterized in Section \ref{sectionouterplanar}. 

For a graph $G=(V,E)$, the {\em line graph} of $G$, denoted $L(G)$, is the graph whose vertex set is the edge set of $G$, and two vertices are adjacent in $L(G)$ if and only if the corresponding edges share an endpoint in $G$. The positive semidefinite maximum nullity of line graphs is studied in \cite{linegraphs}. In particular, they state that if $G$ is a connected graph on $n$ vertices and $m$ edges, then $\Mp(L(G))\geq m-n+2.$ Using this together with the fact that $T(L(G))\geq \left\lceil\frac{m}{2}\right\rceil,$ we get the following theorem. 

\begin{thm}
    If $G$ is a connected graph on $n$ vertices and $m\geq 2n-3$ edges, then $T(L(G))\leq\Mp(L(G)).$
\end{thm}


The next theorem shows that the conjecture  $T(G)\leq \Mp(G)$ holds true for any graph with a large number of edges. 
\begin{thm}
	Let $G$ be a graph on $n$ vertices and $m\leq \frac{3n}{2}-4$ edges. Then $T(\overline{G})\leq \Mp(\overline{G})$.  
\end{thm}

\begin{proof}
	Let $T$ be a spanning tree of $G$. Then $\Mp(\overline{T})\geq n-3$ (see \cite[Theorem 3.16 and Corollary 3.17]{AIM08}). Note that $G$ can be obtained from $T$ by adding at most $m-(n-1)$ edges, so  $\overline{G}$ can be obtained from $\overline{T}$ by deleting at most $m-(n-1)$ edges. Since edge deletion decreases the positive semidefinite maximum nullity by at most 1, then \[\Mp(\overline{G})\geq \Mp(\overline{T})-(m-(n-1))\geq (n-3)-(m-(n-1))\geq \frac{n}{2}, \] where the last inequality follows from the fact that $m\leq \frac{3n}{2}-4$. Since $\Mp(G)$ is an integer, by Theorem \ref{treecovernover2}, we have that $\Mp(\overline{G})\geq T(\overline{G})$. 
\end{proof}

\begin{defn}{\rm
		For a graph $G=(V,E)$, let $G^{\triangle}$ be the graph constructed from $G$ such that for each edge $e=\{u,v\}\in E$, add a new vertex $w_e$ where $w_e$ is adjacent to exactly $u$ and $v$. The vertices $w_e$ are called {\em edge-vertices} of $G^{\triangle}$.}
\end{defn}

\begin{figure}[h!]
	\begin{center}
			\begin{tikzpicture}[scale=1]
			\vertex (v0) at (1.5144, 3.0) {};
			\vertex (v1) at (-1,0.5784) {};
			\vertex (v2) at (4.0,0.6667){};
			\vertex (v3) at (1.4946,1.7255){};
			\Bvertex (v4) at (0.1791, 2.1373) {};
			\Bvertex (v5) at (2.8884,2.1078){};
			\Bvertex (v6) at (1.4946,0.0){};
			\Bvertex (v7) at (1.0939,0.9412){};
			\Bvertex (v8) at (1.8195, 0.9608) {};
			\Bvertex (v9) at (1.7762, 2.3039) {};
			
			\draw[] (v0)to (v1);
			\draw [] (v0) to (v2);
			\draw [] (v0) to (v3);
			\draw [] (v0) to (v4);
			\draw [] (v0) to (v5);
			\draw [] (v0) to (v9);
			\draw [] (v1) to (v2);
			\draw [] (v1) to (v3);
			\draw [] (v1)to (v4);
			\draw [] (v1) to (v6);
			\draw [] (v1) to (v7);
			\draw[] (v2)to (v3);
			\draw [] (v2) to (v5);
			\draw [] (v2)to (v6);
			\draw [] (v2)to (v8);
			\draw [] (v3)to (v7);
			\draw [] (v3)to (v8);
			\draw [] (v3) to (v9);
			\end{tikzpicture}
			\caption{$K_4^{\triangle}$, where the edge-vertices are black. }
		\end{center}
		\end{figure}
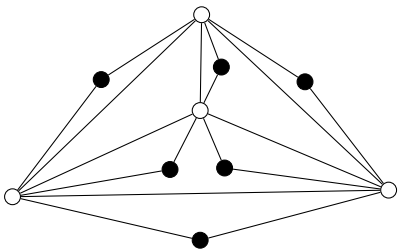

\begin{thm}
	For a connected graph $G=(V,E)$ on $n$ vertices and $m$ edges, $T(G^{\triangle})\leq \Mp(G^{\triangle}).$
\end{thm}

\begin{proof}
	We show that $\mrp(G^{\triangle})=\alpha(G^{\triangle})$ and then apply Proposition \ref{treeindep}. It is always the case that a connected graph $H$ has $\alpha(H)\leq \mrp(H)$ (see Corollary 2.7 in \cite{booth08}), so we show that $\mrp(G^{\triangle})\leq\alpha(G^{\triangle}).$ Let $B$ be the vertex-edge incidence matrix of $G$, and let $X=\begin{pmatrix} I_m\\ B \end{pmatrix}$, where $I_m$ is the $m\times m$ identity matrix. Then $XX^T=\begin{pmatrix}
	I_m & B^T\\B & BB^T\\
	\end{pmatrix}\in \mathcal{S}_+(G^{\triangle}),$ where the first $m$ rows and columns are indexed by the edge-vertices and the last $n$ rows and columns are indexed by the vertices in $V$. Note that the set of edge-vertices of $G^{\triangle}$ is an independent set of size $m$ and that the rank of $XX^T$ is $m$. So $\mrp(G^{\triangle})\leq m\leq \alpha(G),$ and therefore $\mrp(G^{\triangle})=\alpha(G^{\triangle})$. By Proposition \ref{treeindep}, $\TC(G^{\triangle})\leq m+n-\mrp(G^{\triangle})=\Mp(G^{\triangle}).$
\end{proof}

The {\em tree-width} of a graph $G$, denoted $\tw(G)$, is a widely studied parameter, and there are multiple ways in which it is defined. Here we define the tree-width in terms of chordal completions. A graph is {\em chordal} if it has no induced cycle on four or more vertices. If $G$ is a subgraph of $H$ such that $V(G)=V(H)$ and $H$ is chordal, then $H$ is called a {\em chordal completion} of $G$. The tree-width of $G$ is defined as \[\tw(G)=\min\{\omega(H)-1| H \text{ is a chordal completion of }G\}.\]
\begin{prop}
	Let $G$ be a graph on $n$ vertices with $\tw(G)\leq \frac{n-4}{2}$. Then $T(\overline{G})\leq \Mp(\overline{G}).$ 
\end{prop}

\begin{proof}
	If $\tw(G)\leq k$, then $\mrp(\overline{G})\leq k+2$ \cite{SH11}, i.e., $\Mp(\overline{G})\geq n-k-2. $ For $k=\frac{n-4}{2},$ it follows that $\Mp(\overline{G})\geq \frac{n}{2}$. Since $\Mp(\overline{G})$ is an integer, we have that $\Mp(\overline{G}) \geq \left \lceil \frac{n}{2}\right \rceil \geq T(\overline{G})$, where the last inequality follows from Theorem \ref{treecovernover2}. 
\end{proof}

A $k-$tree is constructed inductively by starting with a complete graph $K_{k+1}$ and at each step a new vertex is added and this vertex is adjacent to exactly $k$ vertices in an existing $K_k$.  A graph $G$ that is a $2-$tree is outerplanar by definition, so $T(G)=\Mp(G)$. Observation \ref{treelessthan3} and Theorem \ref{oddktree} show that $T(G)\leq \Mp(G)$ for $3$-trees and $5$-trees.

\begin{obs}{\rm \label{treelessthan3}If $G$ is a graph with $T(G)\leq 3,$ then it holds that $T(G)\leq \Mp(G)$. This is because of the fact that $\Mp(G)=1$ implies $G$ is a tree \cite{hvd} (so $T(G)=1$), and the facts that $T(G)\leq \Zp(G)$ \cite{EEHHH} (where $\Zp(G)$ is the positive semidefinite zero forcing number of a graph, defined in \cite{zfparam}) and $\Mp(G)=2$ implies $\Zp(G)=2 \cite{EEHHH}$ (so $T(G)\leq 2).$}\end{obs}

\begin{thm}\label{oddktree}{\rm \cite{graphcoverings}}
    Let $G$ be a $k-$tree with $k$ odd. If $G$ is a $k-$ tree, then $T(G)=\frac{k+1}{2}.$
\end{thm}

\begin{cor}
    For $k\in \{3,5\},$ $T(G)\leq \Mp(G).$
\end{cor}

\section{$T(G)$ of graphs with girth at least 5}\label{sectionbounds}For many graphs, the tree cover number is much lower than $\left\lceil\frac{n}{2}\right\rceil.$ The next theorem improves this bound for graphs with girth at least 5.

\begin{thm}{\label{nthree}}
	Let $G$ be a connected graph on $n\geq 6$ vertices with girth at least 5. Then $T(G)\leq \frac{n}{3}.$
\end{thm}
\begin{proof}
	The proof is by induction on $n$. A connected graph on 6 vertices with girth at least 5 is either a tree, $C_6$, or $C_5$ with a leaf adjacent to one of the vertices on the cycle. In each case, the tree cover number is at most 2, so the theorem holds.   Let $n\geq 7$. If $G$ has a leaf $v$, then $T(G)=T(G-v) \leq \frac{n-1}{3}$. Suppose $G$ has no leaves. Let $P=(x,y,z)$ be an induced path in $G$. We consider the connected components of $G-P$ (see Figure \ref{fig1}).

	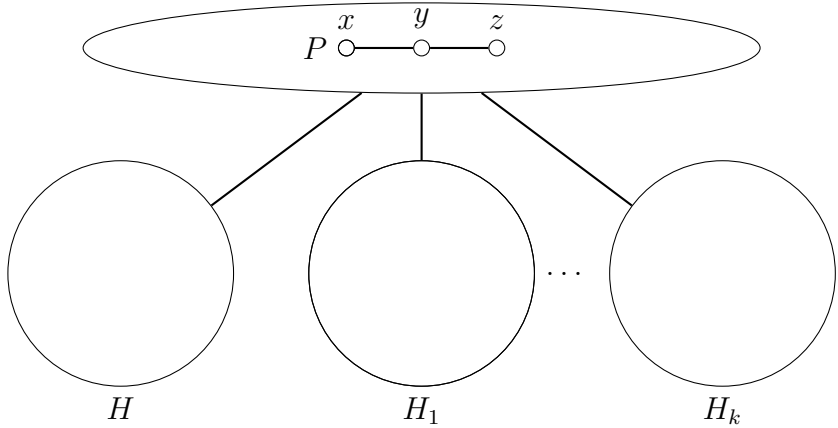
\begin{figure}[h!]
		\begin{center}
			\begin{tikzpicture}[scale=1]
			
			\node [label=above:] [draw, ellipse, minimum height=1.2cm, minimum width=9cm] (elli) at (6.5,6) {};
			
			\node [label=below: $H$  ] [draw, ellipse, minimum height=3cm, minimum width=3cm] (C1) at (2.5,3) {};
			
			\node [label=right:$\hdots$ ] [draw, ellipse, minimum height=3cm, minimum width=3cm] (C2) at (6.5,3) {};
			
			\node [label=below: $H_1$ ] [draw, ellipse, minimum height=3cm, minimum width=3cm] (C2) at (6.5,3) {};
			
			\node [label=below:$H_k$ ] [draw, ellipse, minimum height=3cm, minimum width=3cm] (C3) at (10.5,3) {};
			
			\vertex (x)[label=left: $P$] at (5.5,6) {};
			\vertex (x)[label=above: $x$] at (5.5,6) {};
			\vertex (y) [label=above: $y$] at (6.5,6) {};
			\vertex (z)[label=above: $z$] at (7.5,6) {};
			
			\tikzset{EdgeStyle/.style={-}}
			\Edge(elli)(C1)
			\Edge(elli)(C2)
			\Edge(elli)(C3)
			
			\Edge(x)(y)
			\Edge(z)(y)
			\end{tikzpicture}
			
			\caption{The partition described in proof of Theorem \ref{nthree}, where $H,H_1,...,H_k$ are the connected components of $G-P$.}
			\label{fig1}
		\end{center}
	\end{figure}

	Note that since $G$ has no leaves and no three or four cycles, $G-P$ cannot have an isolated vertex as a connected component. We now show that if $G-P$ has a connected component $H$ with $|H|\in \{3,4,5\}$, then the theorem holds. Suppose $G-P$ has a connected component $H$ of order 3 (i.e., $H$ is a path on three vertices). Note that $G-H$ is a connected graph (since the remaining components of $G-P$ are all connected to $P$), so if $|G-H|\geq 6$, by applying the hypothesis to $G-H$ and covering $H$ with a path to get that $T(G)\leq 1+\frac{n-3}{3}=\frac{n}{3}$. Otherwise, $|G-H|=5$ since $n\geq 7$ and $G-P$ does not have an isolated vertex as a component, so $G-H-P=K_2$. By assumption $G$ has no leaves and no three or four cycles, so $G-H=C_5$, $G$ is one of the two graphs shown in Figure \ref{fig:gggg} and the theorem holds.
	
	Suppose that $G-P$ has a connected component $H$ of order 4. Then $H$ is a tree. If $|G-H|\geq 6$, then $T(G)\leq 1+\frac{n-4}{3}=\frac{n-1}{3}$. If $G-H=P$, then $T(G)=2$, $n=7$, and the theorem holds. Otherwise $G-H=C_5$, $T(G)\leq 3$ (since $G-H$ may be covered with 2 trees and $H$ is a tree, $n=9$, and the theorem holds. 
	
	Consider $G-P$ having a connected component $H$ of order 5. Then $H$ is either a tree or $H=C_5$. Assume first that $H$ is a tree. If $|G-H|\geq 6,$ then $T(G)\leq 1 + \frac{n-5}{3}=\frac{n-2}{3}.$ If $G-H=P,$ then $T(G)=2, n=8,$ and the theorem holds. Otherwise, $G-H=C_5$, $T(G)\leq 3$, $n=10$, and the theorem holds. 
	
	Suppose $H=C_5=(u_1,\ldots.,u_5)$, and without loss of generality, assume that $u_1$ has a neighbor on $P=(x,y,z)$. If $G-H=P,$ then $n=8$ and for $T_1=G[\{u_2,u_3,u_4,u_5\}]$ and $T_2=G[\{x,y,z,u_1\}]$, $\mathcal{T}=\{T_1,T_2\}$ is a tree cover of size 2. Otherwise, for path $P'=(u_2,u_3,u_4,u_5)$, $G-P'$ is a connected graph on at least 6 vertices, so $T(G)\leq 1+\frac{n-4}{3}=\frac{n-1}{3}$.
	
	We may now assume that each component of $G-P$ is $K_2$ or has at least $6$ vertices. If all components of $G-P$ are of order at least 6, then by the induction hypothesis, $T(G)\leq 1+\frac{n-3}{3}=\frac{n}{3}$. Suppose $G-P$ has exactly one component that is $K_2=(u,v)$. Since $G$ has no leaves then each of $u$ and $v$ must be adjacent to a vertex of $P$, and since $G$ has no three or four cycles, then $u$ must be adjacent to $x$ and $v$ must be adjacent to $z$. Furthermore, since $n\geq 7,$ then $G-P$ must have a component $H$ with at least $6$ vertices. Note that $H$ has a vertex that is adjacent to some $r\in \{x,y,z\}$. By adding $r$ to $H$, we partition $G$ into a tree (namely, the tree with vertex set $\{x,y,z,u,v\}\setminus\{r\}$) and connected components of order at least 6. Thus, $T(G)\leq 1+\frac{n-4}{3}=\frac{n-1}{3}.$
	
	Suppose $G-P$ has $s\geq 2$ components that are $K_2$. We first show that the vertices of $P=(x,y,z)$ and the vertices of each $K_2$ can be covered with two trees: recall that since $G$ has no leaves then each endpoint of a $K_2$ must be adjacent to a vertex of $P$, and since $G$ has no three or four cycles, then for each $K_2$, one endpoint must be adjacent to $x$ and the other end must be adjacent to $z$. Let $X$ be the set of endpoints that are adjacent to $x$  and let $Z$ be the set of of endpoints that are adjacent to $z$. Then for $T_1=G[X\cup \{x\}]$ and $T_2=G[Z\cup \{z,y\}]$, $\mathcal{T}=\{T_1, T_2\}$ is a tree cover of size two that covers the vertices of $P$ and the vertices of $G-P$ belonging to a $K_2$. We apply the induction hypothesis to each component of $G-P$ with at least 6 vertices to get that $T(G)\leq 2+\frac{n-3-2s}{3}\leq \frac{n-1}{3}$. \end{proof}
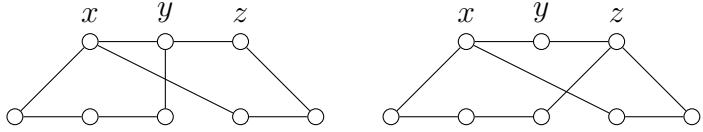
\begin{figure}[h!]
	\begin{center}
		\begin{tikzpicture}[scale=1.0]
		\vertex [label=above:$x$](1) at (0,0) {};
		\vertex [label=above:$y$](2) at (1,0) {};
		\vertex [label=above:$z$](3) at (2,0) {};
		\vertex (4) at (-1,-1){};
		\vertex (5) at (0,-1){};
		\vertex (6) at (1,-1){};
		\vertex(7) at (2,-1){};
		\vertex(8) at (3,-1){};

		\draw (1) to (2);
		\draw (2) to (3);
		\draw(4) to (5);
		\draw(5) to(6);
		\draw (7) to (8);
		\draw(4) to (1);
		\draw(6) to (2);
		\draw(7) to (1); 
		\draw(8) to (3);
		
		\vertex [label=above:$x$](11) at (5,0) {};
		\vertex [label=above:$y$](12) at (6,0) {};
		\vertex [label=above:$z$](13) at (7,0) {};
		\vertex (14) at (4,-1){};
		\vertex (15) at (5,-1){};
		\vertex (16) at (6,-1){};
		\vertex(17) at (7,-1){};
		\vertex(18) at (8,-1){};

		\draw (11) to (12);
		\draw (12) to (13);
		\draw(14) to (15);
		\draw(15) to(16);
		\draw (17) to (18);
		\draw(14) to (11);
		\draw(16) to (13);
		\draw(17) to (11); 
		\draw(18) to (13);
		
		x		\end{tikzpicture}
		\caption{Graphs in Proof of Theorem \ref{nthree}}
		\label{fig:gggg}
	\end{center}
	
\end{figure}
\begin{cor}\label{girth5nover3}
	If $G$ is a connected outerplanar graph on $n$ vertices with girth at least 5, then $\Mp(G)\leq \frac{n}{3}$. 
\end{cor}

Triangle-free graphs are a family of widely studied graphs, so an interesting question is whether or not the bound given in Corollary \ref{girth5nover3} holds when girth is at least 4. The cycle on four vertices demonstrates that the bound no longer holds. However, computations in Sage suggest the next conjecture.
\begin{conj}
	For all connected triangle-free graphs, $T(G)\leq\left\lceil\frac{n}{3}\right\rceil$. 
\end{conj}

\section{$\Mp$ and $T$ for connected outerplanar graphs}\label{sectionouterplanar}
We now turn our attention specifically to connected outerplanar graphs on $n\geq 2$. We have seen that $\Mp(G)=T(G)\leq \left\lceil\frac{n}{2}\right\rceil$ for these graphs, and in this section we characterize graphs that achieve this upper bound. 

Let $\mathcal{F}$ denote the block-clique graphs such that each clique is $K_3$ (see Figure \ref{fig:blockclique}). Observe that every graph in $\mathcal{F}$ has an odd number of vertices. We show that for $n$ odd, the family of graphs whose tree cover number achieves the upper bound is exactly the family $\mathcal{F}$. The family $\mathcal{F}$ also plays a vital role in characterizing outerplanar graphs on an even number of vertices whose tree cover number achieves the upper bound.

We begin by first stating the results that provide this characterization.
\begin{figure}[h!]
	\begin{center}
		\begin{tikzpicture}[scale=1.5]
		\GraphInit[vstyle=Classic]
		\vertex (0) at (1.6971, 3.0) {}; 
		\vertex (1) at (1.0401, 2.979) {}; 
		\vertex (2) at (1.4343, 2.2972) {}; 
		\vertex (3) at (0.6131,1.8252 ) {}; 
		\vertex (4) at (1.3467, 1.5629) {}; 
		\vertex (5) at (2.2993, 2.6119) {}; 
		\vertex (6) at (2.0693, 2.0245) {};
		\vertex (7) at (1.9927, 1.3112) {}; 
		\vertex (8) at (2.6277,1.7203 ) {};
		\vertex (9) at (1.6752, 0.8287) {}; 
		\vertex (10) at (2.2445, 0.6923) {}; 
		\vertex (11) at (0.6569, 1.2378) {}; 
		\vertex (12) at (0.0, 1.4685) {}; 
		\vertex (13) at (0.2847, 0.6818) {}; 
		\vertex (14) at (0.9088,0.5455 ) {}; 
		\vertex (15) at (0.1642, 2.4965) {}; 
		\vertex (16) at (0.8869, 2.5385) {}; 
		\vertex (17) at (2.3869, 0.0) {}; 
		\vertex (18) at (2.8686, 0.3776) {}; 
		\vertex (19) at (2.6606, 1.3007) {}; 
		\vertex (20) at (3.0, 0.8706 ) {}; 
		\vertex (21) at (1.1715,0.3252 ) {}; 
		\vertex (22) at (1.8285, 0.2203) {}; 
		\Edge[](0)(1)
		\Edge[](0)(2)
		\Edge[](1)(2)
		\Edge[](2)(3)
		\Edge[](2)(4)
		\Edge[](2)(5)
		\Edge[](2)(6)
		\Edge[](3)(4)
		\Edge[](3)(11)
		\Edge[](3)(12)
		\Edge[](3)(15)
		\Edge[](3)(16)
		\Edge[](5)(6)
		\Edge[](6)(7)
		\Edge[](6)(8)
		\Edge[](7)(8)
		\Edge[](7)(9)
		\Edge[](7)(10)
		\Edge[](9)(10)
		\Edge[](9)(21)
		\Edge[](9)(22)
		\Edge[](10)(17)
		\Edge[](10)(18)
		\Edge[](10)(19)
		\Edge[](10)(20)
		\Edge[](11)(12)
		\Edge[](11)(13)
		\Edge[](11)(14)
		\Edge[](13)(14)
		\Edge[](15)(16)
		\Edge[](17)(18)
		\Edge[](19)(20)
		\Edge[](21)(22)
		\end{tikzpicture}
	\end{center}
	\caption{A block-clique graph such that each clique is $K_3$.}
	\label{fig:blockclique}
\end{figure}
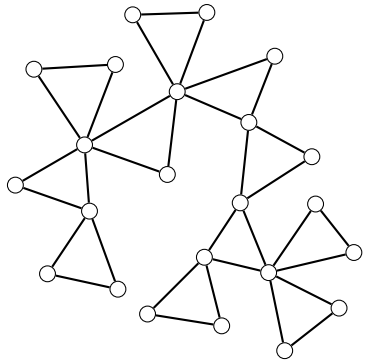

\begin{thm}\label{outerplanarodd}
	Let $G$ be a connected outerplanar graph of odd order $n\geq 3$. Then $T(G)=\left\lceil\frac{n}{2}\right\rceil$ if and only if $G\in \mathcal{F}$. 
\end{thm}

\begin{cor}
	Let $G$ be a connected outerplanar graph of odd order $n$. Then $\Mp(G)=\left\lceil\frac{n}{2}\right\rceil$ if and only if $G\in \mathcal{F}$.
\end{cor}

The only connected graph of order $n=2$ is $K_2,$ which has tree cover number $\frac{n}{2}=1.$
\begin{thm}\label{outerplanareven}
	For a connected outerplanar graph $G=(V,E)$  of even order $n\geq 4$, $T(G)=\frac{n}{2}$ if and only if one of the following holds:
	\begin{enumerate}
		\item[(1)] $G$ is obtained from some $G'\in \mathcal{F}$ by adding one leaf.
		\item[(2)]  $G$ is obtained from some $G_1,G_2\in \mathcal{F}$ by connecting them with a bridge. 
		\item[(3)] $G$ is constructed from the following iterative process: Start with $G^{[0]}\in \{C_4, K_4-e, C_r^{\triangle} (\text{ for some } r\geq 3)\}$. For $i\geq 1$, pick a $v\in V(G^{[i-1]})$ and let $G^{[i]}=G^{[i-1]}\oplus_v K_3$.
	\end{enumerate}
\end{thm}

\begin{cor}
	Let $G=(V,E)$ be a connected outerplanar graph of even order $n$. Then $\Mp(G)=\frac{n}{2}$ if and only if one of (1), (2), (3) of Theorem \ref{outerplanareven} holds.
\end{cor}

For a block-clique graph where each clique is a cycle, note that any two cycles share at most one common vertex. Two blocks are said to be {\em adjacent} if they have one common vertex. A {\em pendant block} is a block that is adjacent to exactly one other block.

\begin{lem}{\rm \cite{graphcoverings}}\label{pendantblock} Any block-clique graph where each clique is a cycle has at least two pendant blocks.  
\end{lem}

To prove Theorem \ref{outerplanarodd}, we use the next two lemmas.

\begin{lem}\label{lem1}
	If $G$ is a connected graph with $n\geq 3$ vertices and $T(G)=\left\lceil\frac{n}{2}\right\rceil,$ then there exist adjacent vertices $u, v\in V(G)$ such that $G'=G[V(G)\setminus\{u,v\}]$ remains connected. Furthermore,  $T(G')=\left\lceil\frac{n-2}{2}\right\rceil.$
	
\end{lem}

\begin{proof}
	By Lemma \ref{deletestar}, we may remove an induced subgraph $H= K_{1,p}$ such that $G-H$ remains connected. First note that if $p\geq 3$, then $T(G)\leq 1 +\left\lceil\frac{n-4}{2}\right\rceil < \left\lceil\frac{n}{2}\right\rceil,$ which is a contradiction, so $p\leq 2$. If $p=1$, then we are done. Suppose $p=2$ (i.e., $H$ is a path $(x,y,z)$). If $x$ and $z$ are both leaves in $G$, then $T(G)=T(G-\{x,z\})\leq \left\lceil\frac{n-2}{2}\right\rceil,$ which is a contradiction to $T(G)=\frac{n}{2}$. So without loss of generality, $x$ has a neighbor in $G-H$, and the theorem holds with $u=y$ and $v=z$. 
	
	It is easy to see that if $T(G')<\left\lceil\frac{n-2}{2}\right\rceil,$ then $T(G)<\left\lceil\frac{n}{2}\right\rceil.$ So, $T(G')=\left\lceil\frac{n-2}{2}\right\rceil.$
\end{proof}

\begin{lem}\label{lem3}
	Let $G=(V,E)$ be a connected graph and suppose $u,v\in V$ are adjacent vertices such that $G'=G[V\setminus\{u,v\}]$ is connected. Let $\mathcal{T'}$ be a minimum tree cover of $G'$, and suppose there exists $w\in V(G')$ such that
	\begin{enumerate}
		\item[(1.)] $V(T_w)=\{w,x\}$
		\item[(2.)] $\exists y\in N(w)\cap N(x)$ such that $N(x)\cap V(T_y)=\{y\}.$ 
	\end{enumerate}
	If $u$ is adjacent to $w$ and $v$  is not adjacent to $w$, then $T(G)\leq T(G')$.
\end{lem}

\begin{proof}
	For $\mathcal{T}=(\mathcal{T'}\setminus\{T_w\cup T_y\})\cup G[\{u,v,w\}]\cup G[V(T_y)\cup\{x\}], \mathcal{T}$ is a tree cover of $G$ of size $T(G')$.
\end{proof}


\noi{\em Proof of Theorem} \ref{outerplanarodd}. Let $G$ be a graph on $n=2k+1$ vertices and first suppose $T(G)=\left\lceil\frac{n}{2}\right\rceil.$ We prove that $G\in \mathcal{F}$ by induction on $k$. If $k=1,$ then $G=K_3$. Let $n=2k+1$ where $k\geq 2$ and suppose that the claim holds for graphs with $2(k-1)+1$ vertices.  By Lemma \ref{lem1}, we can delete an edge $H$ (including the endpoints) such that $G-H$ is connected. Note that since $T(G)=\left\lceil\frac{n}{2}\right\rceil$, then $T(G-H)=\left\lceil\frac{n-2}{2}\right\rceil.$ By the induction hypothesis, $G-H \in \mathcal{F}$ (see the next figure).

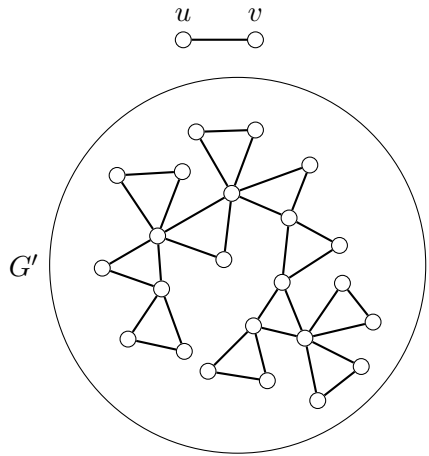
\begin{center}
	\begin{tikzpicture}[scale=1.2]
	\GraphInit[vstyle=Classic]
	
	\node [label=left: $G'$] [draw, ellipse, minimum height=5cm, minimum width=5cm] (elli) at (1.5,1.5) {};
	
	\vertex [label=above: $v$] (v) at (1.6971, 4.0) {};
	\vertex [label=above: $u$] (u) at (0.8971, 4.0) {};
	\vertex (0) at (1.6971, 3.0) {}; 
	\vertex (1) at (1.0401, 2.979) {}; 
	\vertex (2) at (1.4343, 2.2972) {}; 
	\vertex (3) at (0.6131,1.8252 ) {}; 
	\vertex (4) at (1.3467, 1.5629) {}; 
	\vertex (5) at (2.2993, 2.6119) {}; 
	\vertex (6) at (2.0693, 2.0245) {};
	\vertex (7) at (1.9927, 1.3112) {}; 
	\vertex (8) at (2.6277,1.7203 ) {};
	\vertex (9) at (1.6752, 0.8287) {}; 
	\vertex (10) at (2.2445, 0.6923) {}; 
	\vertex (11) at (0.6569, 1.2378) {}; 
	\vertex (12) at (0.0, 1.4685) {}; 
	\vertex (13) at (0.2847, 0.6818) {}; 
	\vertex (14) at (0.9088,0.5455 ) {}; 
	\vertex (15) at (0.1642, 2.4965) {}; 
	\vertex (16) at (0.8869, 2.5385) {}; 
	\vertex (17) at (2.3869, 0.0) {}; 
	\vertex (18) at (2.8686, 0.3776) {}; 
	\vertex (19) at (2.6606, 1.3007) {}; 
	\vertex (20) at (3.0, 0.8706 ) {}; 
	\vertex (21) at (1.1715,0.3252 ) {}; 
	\vertex (22) at (1.8285, 0.2203) {}; 
	\Edge[](0)(1)
	\Edge[](0)(2)
	\Edge[](1)(2)
	\Edge[](2)(3)
	\Edge[](2)(4)
	\Edge[](2)(5)
	\Edge[](2)(6)
	\Edge[](3)(4)
	\Edge[](3)(11)
	\Edge[](3)(12)
	\Edge[](3)(15)
	\Edge[](3)(16)
	\Edge[](5)(6)
	\Edge[](6)(7)
	\Edge[](6)(8)
	\Edge[](7)(8)
	\Edge[](7)(9)
	\Edge[](7)(10)
	\Edge(9)(10)
	\Edge[](9)(21)
	\Edge[](9)(22)
	\Edge[](10)(17)
	\Edge[](10)(18)
	\Edge[](10)(19)
	\Edge[](10)(20)
	\Edge[](11)(12)
	\Edge[](11)(13)
	\Edge[](11)(14)
	\Edge[](13)(14)
	\Edge[](15)(16)
	\Edge[](17)(18)
	\Edge[](19)(20)
	\Edge[](21)(22)
	
	\Edge(u)(v)
	\end{tikzpicture}
\end{center}

Furthermore, by using Lemma \ref{lem1}, $G-H$ has a minimum tree cover, $\mathcal{T}$, such that one tree has exactly one vertex and the remaining trees have exactly two vertices. Let $V(H)=\{u,v\}$. We show that $G\in \mathcal{F}$ by showing 1) if $u$ is adjacent to a vertex $w\in V(G-H),$ then $v$ must also be adjacent to $w$ and 2) $u$ (and therefore $v$) is adjacent to exactly one vertex in $V(G-H)$.

To see 1), suppose $u$ is adjacent to $w\in V(G-H)$ and $v$ is not adjacent to $w$. If $V(T_w)=\{w\}$, then $\mathcal{T}'=(\mathcal{T}\setminus T_w)\cup G[\{w,v,u\}]$ is a tree cover of $G$ of size $\left\lceil\frac{n-2}{2}\right\rceil$, which is a contradiction. Otherwise, $T_w=P_2=(w,x)$ for some $x\in V(G-H)$. Since each edge of $G-H$ belongs to a triangle, then there exists $y\in V(G-H)$ such that $y\in N(w)\cap N(x)$. Since any two triangles in $G-H$ share at most one vertex, it follows from Lemma \ref{lem3} that $T(G)\leq \left\lceil\frac{n-2}{2}\right\rceil$.  (See the next figure.) So $v$ must be adjacent to $w$.

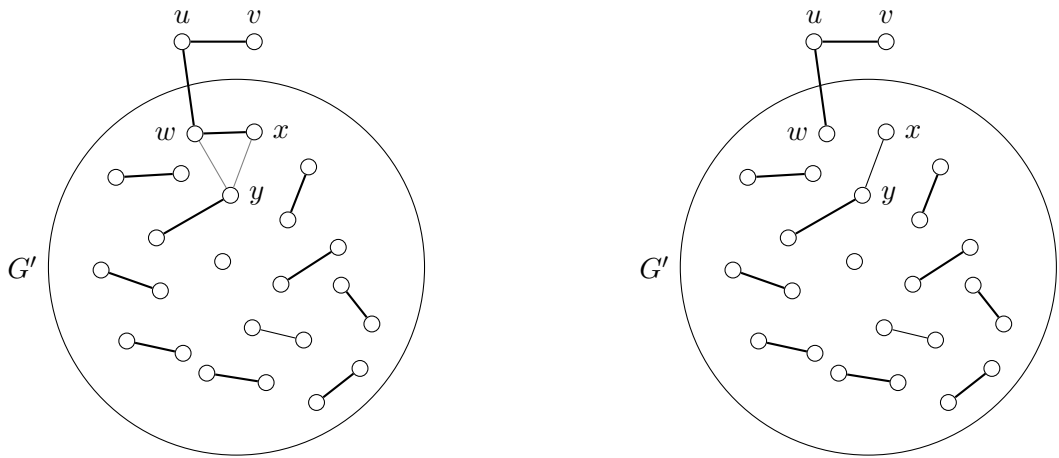
\begin{center}
	\begin{tikzpicture}[scale=1.2]
	\GraphInit[vstyle=Classic]
	
	\node [label=left: $G'$] [draw, ellipse, minimum height=5cm, minimum width=5cm] (elli) at (1.5,1.5) {};
	
	\vertex [label=above: $v$] (v) at (1.6971, 4.0) {};
	\vertex [label=above: $u$] (u) at (0.8971, 4.0) {};
	\vertex [label=right: $x$](0) at (1.6971, 3.0) {}; 
	\vertex [label=left: $w$](1) at (1.0401, 2.979) {}; 
	\vertex [label=right: $y$ ](2) at (1.4343, 2.2972) {}; 
	\vertex (3) at (0.6131,1.8252 ) {}; 
	\vertex (4) at (1.3467, 1.5629) {}; 
	\vertex (5) at (2.2993, 2.6119) {}; 
	\vertex (6) at (2.0693, 2.0245) {};
	\vertex (7) at (1.9927, 1.3112) {}; 
	\vertex (8) at (2.6277,1.7203 ) {};
	\vertex (9) at (1.6752, 0.8287) {}; 
	\vertex (10) at (2.2445, 0.6923) {}; 
	\vertex (11) at (0.6569, 1.2378) {}; 
	\vertex (12) at (0.0, 1.4685) {}; 
	\vertex (13) at (0.2847, 0.6818) {}; 
	\vertex (14) at (0.9088,0.5455 ) {}; 
	\vertex (15) at (0.1642, 2.4965) {}; 
	\vertex (16) at (0.8869, 2.5385) {}; 
	\vertex (17) at (2.3869, 0.0) {}; 
	\vertex (18) at (2.8686, 0.3776) {}; 
	\vertex (19) at (2.6606, 1.3007) {}; 
	\vertex (20) at (3.0, 0.8706 ) {}; 
	\vertex (21) at (1.1715,0.3252 ) {}; 
	\vertex (22) at (1.8285, 0.2203) {}; 
	\draw [thick,black] (0) to (1);
	\draw [gray] (0) to (2);
	\draw [gray](1) to (2);
	\draw [thick, black] (2) to (3);
	\draw [thick, black] (5) to (6);
	\draw [black, thick] (7) to (8);
	\draw  (9) to (10);
	\Edge[](11)(12)
	\Edge[](13)(14)
	\Edge[](15)(16)
	\Edge[](17)(18)
	\Edge[](19)(20)
	\Edge[](21)(22)
	
	\Edge(u)(v)
	\Edge(1)(u)

	\node [label=left: $G'$] [draw, ellipse, minimum height=5cm, minimum width=5cm] (elli) at (8.5,1.5) {};
	
	\vertex [label=above: $v$] (v1) at (8.6971, 4.0) {};
	\vertex [label=above: $u$] (u1) at (7.8971, 4.0) {};
	\vertex [label=right: $x$ ](01) at (8.6971, 3.0) {}; 
	\vertex [label=left: $w$](11) at (8.0401, 2.979) {}; 
	\vertex [label=right: $y$](21) at (8.4343, 2.2972) {}; 
	\vertex (31) at (7.6131,1.8252 ) {}; 
	\vertex (41) at (8.3467, 1.5629) {}; 
	\vertex (51) at (9.2993, 2.6119) {}; 
	\vertex (61) at (9.0693, 2.0245) {};
	\vertex (71) at (8.9927, 1.3112) {}; 
	\vertex (81) at (9.6277,1.7203 ) {};
	\vertex (91) at (8.6752, 0.8287) {}; 
	\vertex (101) at (9.2445, 0.6923) {}; 
	\vertex (111) at (7.6569, 1.2378) {}; 
	\vertex (121) at (7.0, 1.4685) {}; 
	\vertex (131) at (7.2847, 0.6818) {}; 
	\vertex (141) at (7.9088,0.5455 ) {}; 
	\vertex (151) at (7.1642, 2.4965) {}; 
	\vertex (161) at (7.8869, 2.5385) {}; 
	\vertex (171) at (9.3869, 0.0) {}; 
	\vertex (181) at (9.8686, 0.3776) {}; 
	\vertex (191) at (9.6606, 1.3007) {}; 
	\vertex (201) at (10.0, 0.8706 ) {}; 
	\vertex (211) at (8.1715,0.3252 ) {}; 
	\vertex (221) at (8.8285, 0.2203) {}; 
	\draw  (01) to (21);
	\draw [thick, black] (21) to (31);
	\draw [thick, black] (51) to (61);
	\draw [black, thick] (71) to (81);
	\draw  (91) to (101);
	\Edge[](111)(121)
	\Edge[](131)(141)
	\Edge[](151)(161)
	\Edge[](171)(181)
	\Edge[](191)(201)
	\Edge[](211)(221)
	
	\Edge(u1)(v1)
	\Edge(u1)(11)
	\end{tikzpicture}
\end{center}

To see 2), suppose that $v$ and $u$ were adjacent to $x,y\in V(G-H)$. We show that $G$ is not outerplanar.  Since $G-H$ is connected, then there is a path, $P$, in $G-H$ with endpoints $x$ to $y$. Then $G[V(P)\cup\{v,u\}]$ has a $K_4-$minor, which contradicts $G$ being outerplanar. 

We show the converse by induction on $k$. Let $G\in \mathcal{F}$. For $k=1$, $G=K_3,$ so $T(G)=\left\lceil\frac{n}{2}\right\rceil.$ For $k\geq 2$, by Lemma \ref{pendantblock}, $G$ has a pendant block so  $G=G'\oplus_v K_3$ for some $v\in V$, where $G$ is a graph on $n-2$ vertices and $G'\in \mathcal{F}$. It follows from the induction hypothesis and Proposition \ref{cutvertextree} that $T(G)=T(G')+T(K_3)-1=\left\lceil\frac{n-2}{2}\right\rceil+1=\left\lceil\frac{n}{2}\right\rceil$. \qed

Since $T(G)=\left\lceil\frac{n}{2}\right\rceil$ for any graph $G\in \mathcal{F},$ by Lemma \ref{lem1}, we can form a minimum tree cover of $G$ where each tree is $K_2$. This shows that $T(G)\leq P(G)$, and since $T(G)\geq P(G)$ is always true, then it follows that for $G\in \mathcal{F},$ $T(G)= P(G).$ The graphs in $\mathcal{F}$ are special cases of the block-cycle graphs studied in \cite{graphcoverings}.  It is shown in \cite{graphcoverings} that $\ZFN(G)=P(G)$ for all block-cycle graphs. Thus, we have that \[\ZFN(G)=\PC(G)=T(G)=\Mp(G)\leq \Zp(G)\leq \ZFN(G)\] for any $G\in \mathcal{F},$ so all  of the parameters are equal. 

To prove Theorem \ref{outerplanareven}, we use an additional lemma.
\begin{lem}\label{lem2}
	Let $G=(V,E)$ be a connected graph of even order $n$ with $T(G)=\frac{n}{2}$ that satisfies the following conditions:
	\begin{enumerate}			
		\item[(a)] $G$ does not have a bridge.
		\item[(b)] $\delta(G)\geq2$, and if $z\in V$ is a vertex such that $N(z)=\{z',z''\}$, then $z'$ and $z''$ are adjacent.
	\end{enumerate}
	Let $u,v\in V$ be adjacent vertices in $G$ such that $G'=G[V\setminus\{u,v\}]$ remains connected and $T(G')=\frac{n-2}{2}$ . If $G'$ does not have a leaf, then one of the following holds:
	
	\begin{enumerate}
		\item $G'$ satisfies (a) and (b).
		\item $G'$ satisfies (3) of Theorem \ref{outerplanareven}.
		\item $G$ satisfies (3) of Theorem \ref{outerplanareven}.
	\end{enumerate}
\end{lem}

\begin{proof}
	Assume $G'$ has no leaves. Suppose first that $e=\{g_1,g_2\}$ is a bridge in $G'$ (i.e., $G'$ does not satisfy (a)) and let $G_1,G_2$ be the connected components of $G'-e$, where $g_1\in V(G_1)$ and $g_2\in V(G_2)$ . We show that $G$ satisfies (3). By hypothesis, $T(G')=\frac{n-2}{2}$ and by Lemma \ref{treebridge}, $T(G')=T(G_1)+T(G_2)-1.$ It follows that $ |G_i|$ is odd and  $T(G_i)=\left\lceil\frac{|G_i|}{2}\right\rceil$ for $i=1,2$. Note that since $G'$ has no leaves, $|G_i|\geq 3$ for $i=1,2,$ and by Theorem \ref{outerplanarodd}, $G_i\in \mathcal{F}$. 
	
	Let $W$ be the set of vertices in $ V(G')\setminus \{g_1,g_2\}$ that are adjacent to either $u$ or $v$. We first show that for each $w\in W$, $w$ is adjacent to both $u$ and $v$. Without loss of generality, let $u$ have a neighbor $w$ in $W$, suppose $v$ is not adjacent to $w$, and suppose that $w\in V(G_1)$. Let $\mathcal{T'}$ be a tree cover of $G'$ such that each tree has exactly two vertices ($\mathcal{T'}$ is guaranteed by Lemma \ref{lem1}) and let $T_w=G[\{w,x\}]$ be the tree containing $w$. Note that $x\in V(G_1)$ since $w\neq g_1$. Since $G_1\in \mathcal{F},$ $w$ and $x$ have a common neighbor $y$. Let  $V(T_y)=\{y,y'\}$, and note that $y'\notin N(x)$ (if $y'\in V(G_1)$ then this follows from the fact that $G_1\in \mathcal{F}$, and if $y'\in V(G_2)$, then this follows from the fact that $x$ has no neighbor in $V(G_2)$). By Lemma \ref{lem3}, $T(G)\leq T(G')=\frac{n-2}{2},$ which contradicts $T(G)=\frac{n}{2}.$ 
	
	Thus, $v$ is adjacent to $w,$ and this shows that $u$ and $v$ have the same set of neighbors in $V(G')\setminus \{g_1,g_2\}$. Furthermore, since $G$ is outerplanar, it follows that $|W|\leq 1$. If $W=\emptyset$, then $G[\{u,v,g_1,g_2\}]$ is $K_4-e$ (since $u$ and $v$ are not leaves, $G$ is outerplanar, $e$ is not a bridge in $G$, and the neighbors of a degree two vertex in $G$ must be adjacent). Since $G_1,G_2\in \mathcal{F}$, it follows that $G$ satisfies (3) with $G^{[0]}=K_4-e$.
	
	Consider $|W|=1$, let $W=\{w\}$, and without loss of generality, suppose $w\in V(G_1).$ Since $e$ is not a bridge in $G$, then $u$ or $v$ must be adjacent to a vertex in $V(G_2),$ and since $|W|=1$, this vertex must be $g_2$.  Without loss of generality, suppose $u$ is adjacent to $g_2$, and note that $v$ cannot also be adjacent to $g_2$ since $G$ is outerplanar. Suppose first that $N(u)\cap (V(T_{g_2})\setminus\{g_2\})=\emptyset$. If $N(v)\cap (V(T_w)\setminus\{w\})=\emptyset,$ then $(\mathcal{T'}\setminus(T_w\cup T_{g_2}))\cup G[V(T_w)\cup \{v\}]\cup G[V(T_{g_2})\cup \{u\}]$ is a tree cover of $G$ of size $\frac{n-2}{2}$, which contradicts $T(G)=\frac{n}{2}$. Thus, $N(v)\cap (V(T_w)\setminus\{w\})\neq\emptyset,$ so it must be the case that $V(T_w)=\{w,g_1\}$ and $v$ is adjacent to $g_1$. But then $G[\{u,v,w,g_1,g_2\}]$  has a $K_4$ minor (see the next figure), which is a contradiction to $G$ being outerplanar. 
	
	
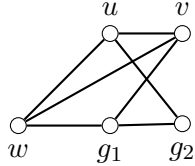
\begin{center}
		\begin{tikzpicture}[scale=1.2]
		\GraphInit[vstyle=Classic]

		\vertex [label=above: $v$] (v) at (1.6971, 4.0) {};
		\vertex [label=above: $u$] (u) at (0.8971, 4.0) {};
		\vertex [label=below: $g_2$](g2) at (1.6971, 3.0) {}; 
		\vertex [label=below: $g_1$](g1) at (.8971, 2.979) {}; 
		\vertex [label=below: $w$](w) at (-0.12, 2.979) {}; 
		
		\Edge(g1)(g2)
		\Edge(u)(v)
		\Edge(u)(w)
		\Edge(v)(w)
		\Edge(w)(g1)
		\Edge(v)(g1)
		\Edge(u)(g2)
		
		\end{tikzpicture}
	\end{center}
	
	So, $N(u)\cap (V(T_{g_2})\setminus\{g_2\})\neq\emptyset$, and it must be the case that $V(T_{g_2})=\{g_1,g_2\}$ and $u$ is adjacent to $g_1$. Note that since $G$ is outerplanar, $v$ is not adjacent to $g_1$ nor $g_2$. It follows that $G$ satisfies  (3) with $G^{[0]}=C_r^{\triangle},$ where $C_r=(u,w,x_1,\ldots,x_j,g_1)$ and  $(w,x_1,\ldots,x_j,g_1)$ is the path between $w$ and $g_1$ in $G_1$ (see the next figure).

	
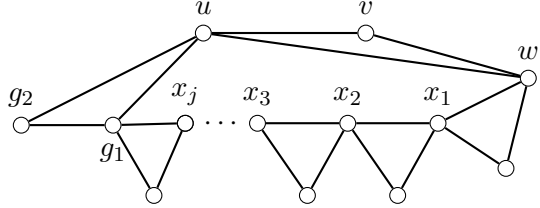
\begin{center}
		\begin{tikzpicture}[scale=1.2]
		\GraphInit[vstyle=Classic]

		\vertex [label=above: $v$] (v) at (3.6971, 4.0) {};
		\vertex [label=above: $u$] (u) at (1.8971, 4.0) {};
		\vertex [label=above: $x_j$](xj) at (1.6971, 3.0) {}; 
		\vertex [label=right: $\cdots$](xj) at (1.6971, 3.0) {}; 
		\vertex [label=above: $x_3$](x3) at (2.4971, 3.0) {}; 
		\vertex [label=above: $x_2$](x2) at (3.4971, 3.0) {}; 
		\vertex [label=above: $x_1$](x1) at (4.4971, 3.0) {}; 
		\vertex [label=above: $w$](w) at (5.4971, 3.5) {};
		\vertex [label=below: $g_1$](g1) at (.8971, 2.979) {}; 
		\vertex [label=above: $g_2$](g2) at (-0.12, 2.979) {}; 
		\vertex [label=below: ](2) at (1.35, 2.2) {};
		\vertex [label=below: ](3) at (3.05, 2.2) {};
		\vertex [label=below: ](4) at (4.05, 2.2) {};
		\vertex [label=below: ](5) at (5.25, 2.5) {};

		\Edge(g1)(g2)
		\Edge(u)(v)
		\Edge(u)(w)
		\Edge(v)(w)
		\Edge(u)(g2)
		\Edge(g1)(xj)
		\Edge(x3)(x2)
		\Edge(x2)(x1)
		\Edge(x1)(w)
		\Edge(u)(g1)
		\Edge(g1)(2)
		\Edge(xj)(2)
		\Edge(x3)(3)
		\Edge(x2)(3)
		\Edge(x2)(4)
		\Edge(x1)(4)
		\Edge(x1)(5)
		\Edge(w)(5)

		\end{tikzpicture}
	\end{center}
	
	Suppose now that $G'$ does not satisfy (b). We show that $G'$ satisfies (3). Since $G'$ does not satisfy (b), then there exists a vertex $z\in V(G')$ of degree 2  whose neighbors $z'$ and $z''$ are not adjacent. By contracting the edge $\{z,z'\}$, we obtain a graph $H$ from $G'$ on $n-3$ vertices with $T(H)=\left\lceil\frac{n-3}{2}\right\rceil$. So, $H\in \mathcal{F}$. Then for the triangle  $(z',z'', y)$ in $H$, it follows that $(z',z,z'',y)$ is a 4 cycle in $G'$, and $G'$  satisfies (3) of Theorem \ref{outerplanareven} with $G^{[0]}=C_4$. 
\end{proof}

{\em Proof of Theorem} \ref{outerplanareven}. Let $G$ be a graph on $n=2k$ vertices and first suppose $T(G)=\frac{n}{2}$. If $G$ has a leaf $v$, then $T(G)=T(G-v)$, and $G-v$ is in $\mathcal{F}$ by Theorem  \ref{outerplanarodd}. Thus (1) holds. If $G$ has a bridge $e$ and the connected components of $G-e$ are $G_1$ and $G_2$, then by Lemma \ref{treebridge}, $T(G)=T(G-e)-1=T(G_1)+T(G_2)-1.$ Note that $|G_1|$ and $|G_2|$ must both be even or both are odd since $n$ is even. If both are even then $T(G)=T(G_1)+T(G_2)-1\leq \frac{|G_1|}{2}+\frac{|G_2|}{2}-1=\frac{n}{2}-1$, which contradicts $T(G)=\frac{n}{2}$. Thus, $|G_1|$ and $|G_2|$ are both odd, and $\frac{n}{2}=T(G)=T(G_1)+T(G_2)-1\leq \left\lceil\frac{|G_1|}{2}\right\rceil+\left\lceil\frac{|G_2|}{2}\right\rceil-1=\frac{|G_1|+1}{2}+\frac{|G_2|+1}{2}-1=\frac{n}{2}.$ It follows that $T(G_i)=\left\lceil\frac{|G_i|}{2}\right\rceil$ for $i=1,2$, so (2) holds. Suppose $G$ can be obtained from some graph $G'$ by subdividing an edge of $G'$. Since subdividing an edge does not change the tree cover number, then $\frac{n}{2}=T(G)=T(G')$ and by Theorem \ref{outerplanarodd}, $G'\in \mathcal{F}$. Note that subdividing an edge of a graph in $\mathcal{F}$ results in a graph in (3) with $G^{[0]}=C_4$, so $G$ satisfies (3).

We may now assume that $\delta(G)\geq 2$, $G$ does not have a bridge, and for each $v\in V$ with $\deg(v)=2$, the neighbors of $v$ are adjacent. For the remainder of the proof, $u$ and $v$ are the adjacent vertices from Lemma \ref{lem1} such that $G'=G[V(G)\setminus \{u,v\}]$ is connected and $T(G')=\frac{n-2}{2}$. We consider two cases, $G'$ has a leaf and $G'$ does not have a leaf.

Case 1. Suppose $G'$ has a leaf $\ell$ and let $\ell'$ be its neighbor. We show $G$ satisfies (3). We first show that $u$ and $v$ have the same set of neighbors in $V(G')\setminus\{l,l'\}$. Note that $T(G'-\ell)=T(G')=\left\lceil\frac{n-3}{2}\right\rceil$ and by Theorem \ref{outerplanarodd}, $G'-\ell$ is in $\mathcal{F}$. Suppose $u$ has a neighbor $w$ in $V(G')\setminus\{l,l'\}$ and $v$ is not adjacent to $w$. Let $\mathcal{T'}$ be a tree cover of $G'$ such that each tree has exactly two vertices ($\mathcal{T'}$ is gauranteed by Lemma \ref{lem1}) and let $T_w=\{w,x\}$ be the tree containing $w$. Since $G'-\ell\in \mathcal{F},$ then $w$ and $x$ have a common neighbor $y$ such that $V(T_y)=\{y,z\}\in \mathcal{T'}$ and $z\notin N(x)$. By Lemma \ref{lem3}, $T(G) \leq \frac{n-2}{2}$, contradicting $T(G)=\frac{n}{2}$. Therefore $u$ and $v$ have the same set of neighbors in $V(G')\setminus\{\ell,\ell'\}$, and since $G$ is outerplanar, this set has cardinality at most one. 

Since $G$ has no leaves, we may assume that $u$ is adjacent to $\ell$. Suppose first that $v$ is also adjacent to $\ell$. Since $\{\ell,\ell'\}$ is not a bridge in $G$, then either $u$ or $v$ must have a neighbor in $G'-\ell$, and since $G$ is outerplanar, $u$ and $v$ cannot both have a neighbor in $G'-\ell$ (since if we contract each edge of $G'-\ell$ to obtain a single vertex $t$, then the graph induced on $\{u,v,\ell,t\}$ would form a $K_4$). Without loss of generality, we let $u$ have a neighbor $w \in V(G'-\ell)$. Since $u$ and $v$ have the same neighbors in $V(G')\setminus \{\ell,\ell'\},$ then $w=\ell'$, $G[\{u,v,\ell,\ell'\}]$ is $K_4-e$, and $G$ satisfies (3) with $G^{[0]}=K_4-e.$  

Assume $v$ is not adjacent to $\ell$. Since $\ell$ has degree 2 in $G$, by hypothesis $u$ is adjacent to $\ell'$, and since $G$ has no leaves, $v$ has a neighbor $w\in V(G'-\ell)$. If $w\neq \ell'$, we have already seen that $u$ must also be adjacent to $w$ and that $N(u)\cap (V(G)\setminus \{\ell,\ell'\})=N(v)\cap (V(G)\setminus \{\ell,\ell'\})=\{w\}.$ Also note that if $w\neq \ell'$, then $v$ cannot also be adjacent to $\ell'$ since $G$ is outerplanar, so $N(u)=\{v,\ell,\ell',w\}$ and $N(v)=\{u,w\}.$ To see that $G$ satisfies (3), let $(\ell',x_1,\ldots,x_j,w)$ be the shortest path from $\ell'$ to $w$ in $G'$ (see the next figure). Since $G'-\ell \in \mathcal{F}$, it follows that $G$ satisfies (3) with $G^{[0]}=C_r^{\triangle}$ and $C_r=(u,\ell,\ell',x_1,\ldots,x_j,w)$. 

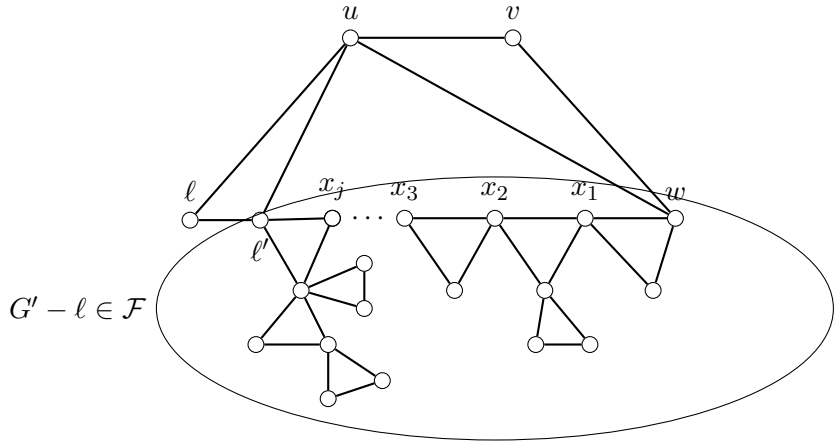
\begin{center}
	\begin{tikzpicture}[scale=1.2]
	\GraphInit[vstyle=Classic]
	
	\node [label=left: $G'-\ell\in \mathcal{F}$] [draw, ellipse, minimum height=3.5cm, minimum width=9cm] (elli) at (3.4971, 2.0) {};

	\vertex [label=above: $v$] (v) at (3.6971, 5.0) {};
	\vertex [label=above: $u$] (u) at (1.8971, 5.0) {};
	\vertex [label=above: $x_j$](xj) at (1.6971, 3.0) {}; 
	\vertex [label=right: $\cdots$](xj) at (1.6971, 3.0) {}; 
	\vertex [label=above: $x_3$](x3) at (2.4971, 3.0) {}; 
	\vertex [label=above: $x_2$](x2) at (3.4971, 3.0) {}; 
	\vertex [label=above: $x_1$](x1) at (4.4971, 3.0) {}; 
	\vertex [label=above: $w$](w) at (5.4971, 3.) {};
	\vertex [label=below: $\ell'$](g1) at (.8971, 2.979) {}; 
	\vertex [label=above: $\ell$](g2) at (0.12, 2.979) {}; 
	\vertex [label=below: ](2) at (1.35, 2.2) {};
	\vertex [label=below: ](3) at (3.05, 2.2) {};
	\vertex [label=below: ](4) at (4.05, 2.2) {};
	\vertex [label=below: ](5) at (5.25, 2.2) {};
	\vertex [label=below: ](21) at (.85, 1.6) {};
	\vertex [label=below: ](22) at (1.65, 1.6) {};
	\vertex [label=below: ](222) at (1.65, 1.0) {};
	\vertex [label=below: ](2222) at (2.25, 1.2) {};
	\vertex [label=below: ](22222) at (2.05, 2.5) {};
	\vertex [label=below: ](222222) at (2.05, 2.0) {};
	\vertex [label=below: ](44) at (4.55, 1.6) {};
	\vertex [label=below: ](444) at (3.95, 1.6) {};

	\Edge(g1)(g2)
	\Edge(u)(v)
	\Edge(u)(w)
	\Edge(v)(w)
	\Edge(u)(g2)
	\Edge(g1)(xj)
	\Edge(x3)(x2)
	\Edge(x2)(x1)
	\Edge(x1)(w)
	\Edge(u)(g1)
	\Edge(g1)(2)
	\Edge(xj)(2)
	\Edge(x3)(3)
	\Edge(x2)(3)
	\Edge(x2)(4)
	\Edge(x1)(4)
	\Edge(x1)(5)
	\Edge(w)(5)
	\Edge(2)(21)
	\Edge(2)(22)
	\Edge(21)(22)
	\Edge(22)(222)
	\Edge(22)(2222)
	\Edge(222)(2222)
	\Edge(2)(22222)
	\Edge(2)(222222)
	\Edge(22222)(222222)
	\Edge(4)(44)
	\Edge(4)(444)
	\Edge(44)(444)

	\end{tikzpicture}
\end{center}

If $v$ is adjacent to $\ell'$, then $u$ and $v$ share the same set of neighbors in $V(G')\setminus \{\ell\}$, and since $G$ is outerplanar we must have that $N(u)=\{v,\ell,\ell'\}, N(v)=\{u,\ell'\}$ (see the next figure). Thus, $G$ satisfies (3) with $G^{[0]}=K_4-e$.

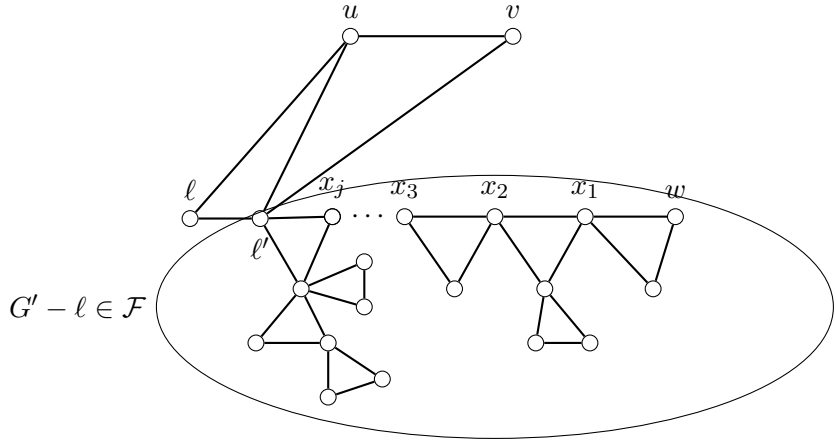
\begin{center}
	\begin{tikzpicture}[scale=1.2]
	\GraphInit[vstyle=Classic]
	
	\node [label=left: $G'-\ell\in \mathcal{F}$] [draw, ellipse, minimum height=3.5cm, minimum width=9cm] (elli) at (3.4971, 2.0) {};

	\vertex [label=above: $v$] (v) at (3.6971, 5.0) {};
	\vertex [label=above: $u$] (u) at (1.8971, 5.0) {};
	\vertex [label=above: $x_j$](xj) at (1.6971, 3.0) {}; 
	\vertex [label=right: $\cdots$](xj) at (1.6971, 3.0) {}; 
	\vertex [label=above: $x_3$](x3) at (2.4971, 3.0) {}; 
	\vertex [label=above: $x_2$](x2) at (3.4971, 3.0) {}; 
	\vertex [label=above: $x_1$](x1) at (4.4971, 3.0) {}; 
	\vertex [label=above: $w$](w) at (5.4971, 3.) {};
	\vertex [label=below: $\ell'$](g1) at (.8971, 2.979) {}; 
	\vertex [label=above: $\ell$](g2) at (0.12, 2.979) {}; 
	\vertex [label=below: ](2) at (1.35, 2.2) {};
	\vertex [label=below: ](3) at (3.05, 2.2) {};
	\vertex [label=below: ](4) at (4.05, 2.2) {};
	\vertex [label=below: ](5) at (5.25, 2.2) {};
	\vertex [label=below: ](21) at (.85, 1.6) {};
	\vertex [label=below: ](22) at (1.65, 1.6) {};
	\vertex [label=below: ](222) at (1.65, 1.0) {};
	\vertex [label=below: ](2222) at (2.25, 1.2) {};
	\vertex [label=below: ](22222) at (2.05, 2.5) {};
	\vertex [label=below: ](222222) at (2.05, 2.0) {};
	\vertex [label=below: ](44) at (4.55, 1.6) {};
	\vertex [label=below: ](444) at (3.95, 1.6) {};

	\Edge(g1)(g2)
	\Edge(u)(v)
	\Edge(v)(g1)
	\Edge(u)(g2)
	\Edge(g1)(xj)
	\Edge(x3)(x2)
	\Edge(x2)(x1)
	\Edge(x1)(w)
	\Edge(u)(g1)
	\Edge(g1)(2)
	\Edge(xj)(2)
	\Edge(x3)(3)
	\Edge(x2)(3)
	\Edge(x2)(4)
	\Edge(x1)(4)
	\Edge(x1)(5)
	\Edge(w)(5)
	\Edge(2)(21)
	\Edge(2)(22)
	\Edge(21)(22)
	\Edge(22)(222)
	\Edge(22)(2222)
	\Edge(222)(2222)
	\Edge(2)(22222)
	\Edge(2)(222222)
	\Edge(22222)(222222)
	\Edge(4)(44)
	\Edge(4)(444)
	\Edge(44)(444)

	\end{tikzpicture}
\end{center}

Case 2: Suppose $G'$ does not have a leaf. We prove this case by induction on $n$. Let $n=6$. Then $G'$ is a graph on four vertices with tree cover number two. Since $G'$ does not have a leaf, then $G'$ is $K_4-e$ or $C_4$. 

Suppose first that $G'=C_4$. If $u$ has a neighbor $w\in V(C_4)$ and $v$ is not adjacent to $w$, then for $T_1=G[\{u,v,w\}]$ and $T_2=G[V(C_4)\setminus\{w\}]$, $\{T_1,T_2\}$ is a tree cover of $G$ of size 2, contradiction $T(G)=3$. So $u$ and $v$ have the same set of neighbors in $V(C_4)$, and since $G$ is outerplanar, $u$ and $v$ have exactly one neighbor in $V(C_4)$ (see next figure), and (3) holds.

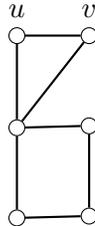
\begin{center}
	\begin{tikzpicture}[scale=1.2]
	\GraphInit[vstyle=Classic]

	\vertex [label=above: $v$] (v) at (1.6971, 4.0) {};
	\vertex [label=above: $u$] (u) at (0.8971, 4.0) {};
	\vertex [label=below: ](2) at (1.6971, 3.0) {}; 
	\vertex [label=below: ](1) at (.8971, 2.979) {};
	\vertex [label=below: ](3) at (1.6971, 2.0) {}; 
	\vertex [label=below:](4) at (.8971, 1.979) {};

	\Edge(1)(2)
	\Edge(2)(3)
	\Edge(3)(4)
	\Edge(4)(1)
	\Edge(u)(1)
	\Edge(v)(1)
	\Edge(u)(v)
	
	\end{tikzpicture}
\end{center}
Consider $G'=K_4-e$. It is well known that an outerplanar graph on $n$ vertices has at most $2n-3$ edges (this can be proven by deleting a vertex of degree two and using induction on $n$). Thus $G$ has at most nine edges. Since there are five edges in $K_4-e$ and one edge between $u$ and $v$, there are at most three edges between the sets $\{u,v\}$ and $V(K_4-e)$, so either $u$ or $v$ has degree two (since $G$ has no leaves). Suppose $N(u)=\{v,w\}$ for some $w\in V(K_4-e)$. By hypothesis, $v$ and $w$ are adjacent. Note that since $G$ has at most nine edges, $v$ can have at most one additional neighbor.  Suppose $v$ has an additional neighbor in $V(K_4-e)$. Then $G$ is one of the graphs given in the next figure, and $T(G)=2$, contradicting $T(G)=3$. Thus, $v$ has no additional neighbors, and $G$ satisfies (3) with $G^{[0]}=K_4-e$.

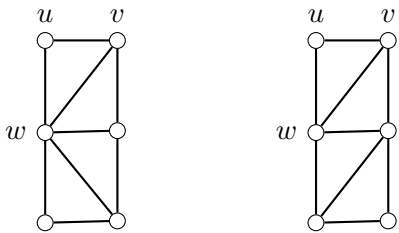
\begin{center}
	\begin{tikzpicture}[scale=1.2]
	\GraphInit[vstyle=Classic]

	\vertex [label=above: $v$] (v1) at (4.6971, 4.0) {};
	\vertex [label=above: $u$] (u1) at (3.8971, 4.0) {};
	\vertex [label=below: ](21) at (4.6971, 3.0) {}; 
	\vertex [label=left: $w$ ](11) at (3.8971, 2.979) {};
	\vertex [label=below: ](31) at (4.6971, 2.0) {}; 
	\vertex [label=below:](41) at (3.8971, 1.979) {};

	\Edge(11)(21)
	\Edge(21)(31)
	\Edge(31)(41)
	\Edge(41)(11)
	\Edge(u1)(11)
	\Edge(v1)(11)
	\Edge(u1)(v1)
	\Edge(21)(41)
	\Edge(v1)(21)

	\vertex [label=above: $v$] (v) at (1.6971, 4.0) {};
	\vertex [label=above: $u$] (u) at (0.8971, 4.0) {};
	\vertex [label=below: ](2) at (1.6971, 3.0) {}; 
	\vertex [label=left:$w$ ](1) at (.8971, 2.979) {};
	\vertex [label=below: ](3) at (1.6971, 2.0) {}; 
	\vertex [label=below:](4) at (.8971, 1.979) {};

	\Edge(1)(2)
	\Edge(2)(3)
	\Edge(3)(4)
	\Edge(4)(1)
	\Edge(u)(1)
	\Edge(v)(1)
	\Edge(u)(v)
	\Edge(1)(3)
	\Edge(v)(2)
	
	\end{tikzpicture}
\end{center}

Let $n\geq 8$. Since $G'$ has no leaves, by Lemma \ref{lem2} we either have that $G$ satisfies (3) (in which case the proof is complete), $G'$  has no bridge and is not a subdivision, or $G'$ satisfies (3). 

Suppose that $G'$  has no bridge and is not a subdivision.  We show that $G'$ satisfies (3). Let $G''$ be the graph obtained from $G'$ after one more application of Lemma \ref{lem1}. If $G''$ has a leaf, $G'$ satisfies (3) by case 1. If $G''$ does not have a leaf, then by the induction hypothesis $G'$ satisfies (3). 

We now use the fact that $G'$ satisfies (3) to show that $G$ satisfies (3) by showing that $N(u)=\{v,w\}$ and $N(v)=\{u,w\}$, for some $w\in V(G')$ (i.e., $G=K_3\oplus_w G'$). Since $u$ is not a leaf, let $w\in V(G')$ be a neighbor of $u$ and suppose first that $v$ is not adjacent to $w$. We show that this contradicts $T(G)=\frac{n}{2}$. Let $\mathcal{T'}$ be a minimum tree cover of $G'$ with each tree having exactly two vertices and let $V(T_w)=\{w,x\}.$ We consider two cases, there exists $y\in N(w)\cap N(x)$ and $ N(w)\cap N(x)=\emptyset.$  Let $y\in N(w)\cap N(x)$ and let $V(T_y)=\{y,z\}$. If $x$ is not adjacent to $z$, by Lemma \ref{lem3}, $T(G)\leq \frac{n-2}{2}$, so $x$ is adjacent to $z$ and $G[\{w,x,y,z\}]$ is $K_4-e$. Since $G$ is outerplanar and $v$ is not adjacent to $w$, then it can be seen by examination that $G[\{u,v,w,x,y,z\}]$ can be covered with two trees, contradicting $T(G)=\frac{n}{2}$.

Consider $N(w)\cap N(x)=\emptyset$. Note that if $G'$ satisfies (3) with $G^{[0]}\in \{K_4-e, C_r^{\triangle}\}$, then every edge of $G'$ would belong to a triangle, so $N(w)\cap N(x)=\emptyset$ implies that $G'$ satisfies (3) with $G^{[0]}=C_4$. Furthermore, every edge of $G'$ that is not an edge of $C_4$ belongs to a triangle, so $\{w,x\}$ is an edge on $C_4$. Since $v$ is not adjacent to $w$, we may cover $G[V(C_4)\cup\{u,v\}]$ with two trees, contradicting $T(G)=\frac{n}{2}$. So, $v$ must be adjacent to $w$, which shows that $u$ and $v$ have the same set of neighbors on $G'$. Since $G$ is outerplanar, $u$ and $v$ must have exactly one common neighbor in $G'$, which shows that $G$ satisfies (3). 

We now show the converse. The removal of a leaf does not affect the tree cover number of a graph, so if $G$ satisfies (1), then $T(G)=\frac{n}{2}$. If $G$ satisfies (2), then by Lemma \ref{treebridge}, $T(G)=T(G_1)+T(G_2)-1=\left\lceil\frac{|G_1|}{2}\right\rceil +\left\lceil\frac{|G_2|}{2}\right\rceil-1=\frac{n}{2}$.  For a graph $G$ satisfying (3), $\frac{n}{2}=T(G')=T(G)$ since subdividing does not affect tree cover number. Suppose $G$ satisfies (3). If $G\in \{C_4, K_4-e, C_r^{\triangle}\},$ then $T(G)=\frac{n}{2}$. Let $G=G^{[k]}$ for some $k\geq 1$. By Proposition \ref{cutvertextree}, and by induction, $T(G)=T(G^{[k-1]})+T(K_3)-1=\frac{n}{2}$.\qed

		\bibliographystyle{alpha}

	\end{document}